\documentclass[11pt]{article}
\usepackage[dvipsnames]{xcolor}
\usepackage{amsthm, amsmath, amssymb, amsfonts, mathtools, url, booktabs, tikz, setspace, fancyhdr, bm}
\usepackage{geometry}
\usepackage{enumerate}
\usepackage[shortlabels]{enumitem}
\usepackage[babel]{microtype}
\usepackage[english]{babel}
\usepackage{comment}
\usepackage{bbm}
\usepackage{csquotes}

\usepackage{tikz}
\usepackage{graphicx}
\usepackage{float}
\usepackage{soul}
\usetikzlibrary{positioning, arrows.meta, shapes.geometric}
\usepackage[normalem]{ulem}

\newtheorem{theorem}{Theorem}[section]
\newtheorem{prop}[theorem]{Proposition}
\newtheorem{conjecture}[theorem]{Conjecture}
\newtheorem{lemma}[theorem]{Lemma}

\newtheorem{claim}[theorem]{Claim}

\usetikzlibrary{decorations.pathmorphing}

\theoremstyle{definition}

\newtheorem*{defn-non}{Definition}

\newtheorem{ques}[theorem]{Question}

\newlist{Case}{enumerate}{2}
\setlist[Case, 1]{%
    label           =   {\bfseries Case \arabic*.},
    labelindent=1em ,labelwidth=1.3cm, labelsep*=1em, leftmargin =!
}
\setlist[Case, 2]{%
    label           =   {\bfseries Subcase \arabic{Casei}.\arabic*.},
    labelindent=-1em ,labelwidth=1.3cm, labelsep*=1em, leftmargin =!
}

\newcommand{\ex}{\mathrm{ex}}
\newcommand{\emb}{\mathrm{emb}}
\newcommand{\Aut}{\mathrm{Aut}}
\newcommand{\diam}{\mathrm{diam}}
\newcommand{\cA}{\mathcal{A}}
\newcommand{\cB}{\mathcal{B}}

\newcommand{\cT}{\mathcal{T}}
\newcommand{\cZ}{\mathcal{Z}}

\title{On the Generalized Rational Exponents Conjecture}
 
\author{
Jianfeng Hou\thanks{Center for Discrete Mathematics and Theoretical Computer Science, Fuzhou University, Fuzhou, China. Email: jfhou@fzu.edu.cn.}
\and
Caihong Yang\thanks{School of Mathematics and Physics, China University of Geosciences, Wuhan, China.
Email: yangch@cug.edu.cn.}
}

\begin{document}
\date{}
\maketitle

\begin{abstract}
For fixed graphs $H$ and $F$, let $\ex(n,H,F)$ denote the maximum number of copies of $H$ in an $n$-vertex $F$-free graph. In this note, we prove the generalized rational exponents conjecture, posed by  Gerbner and Palmer, showing that for every rational number $\alpha\ge1$, there exist fixed graphs $H_\alpha$ and $F_\alpha$ such that
\[
\ex(n,H_\alpha,F_\alpha)=\Theta(n^\alpha).
\]
Furthermore,  the counting graph $H_\alpha$ can always be chosen connected with diameter at most $3$. Our argument hinges on a localization--compression--shift framework, which transforms the Bukh--Conlon finite family construction for edges into a generalized Tur\'an problem setting with a single forbidden graph.

\end{abstract}

\section{Introduction}

For a graph $G$, let $v(G)=|V(G)|$ and $e(G)=|E(G)|$. 
Given a finite family $\mathcal{F}$ of graphs, we say that $G$ is 
$\mathcal{F}$-free if it contains no member of $\mathcal{F}$ as a subgraph. 
For a graph $H$, let $N(H,G)$ be the number of copies (non-induced subgraphs) of $H$ in $G$. The \emph{generalized Tur\'an number} of $\mathcal{F}$ is 
defined as
\[
\ex(n,H,\mathcal{F}):=\max\bigl\{ N(H,G) : v(G)=n \text{ and } 
G \text{ is } \mathcal{F}\text{-free} \bigr\}.
\]
When $\mathcal{F}=\{F\}$, we write $\ex(n,H,F)$.  Taking $H=K_2$ recovers the classical \emph{Tur\'an number}:
\[
\ex(n,\mathcal{F})=\ex(n,K_2,\mathcal{F}).
\]

Determining $\ex(n,F)$ for a given graph $F$ is a central problem in extremal graph theory. 
This  was initiated by the classical theorems of Mantel~\cite{Mantel07} and Tur\'an~\cite{T41}, 
which determine the exact value of $\ex(n,F)$ when $F$ is a complete graph. 
Thanks to the Erd\H{o}s--Stone--Simonovits theorem, the asymptotic behaviour of $\ex(n,F)$ 
is now well understood for all non-bipartite $F$. 
For bipartite graphs $F$, often referred to as \emph{degenerate}, it remains largely open. Even to date there are still few degenerate graphs the asymptotics of whose Tur\'an numbers are known.  Prominent classical results include the K\H{o}v\'ari--S\'os--Tur\'an~\cite{KST54} bound for complete bipartite graphs 
and the even-cycle theorem of Bondy and Simonovits~\cite{BS74}. 
We refer the reader to~\cite{Bukh2015,Bukh2024,CO2021,FS13,HeMaYang2021,Janzer2023,JanzerSudakov2024,MaYang2023C4,MaYang2025TriangleFourCycle} for some recent progress on this topic.

One of the basic inverse problems in degenerate extremal graph theory is the rational exponents conjecture of Erd\H{o}s and Simonovits~\cite{Erdos1981}.

\begin{conjecture} [Rational exponents conjecture]\label{Conj:Rational-Exponents-Conjecture}
For every rational number $\beta\in(1,2)$, there exists a bipartite graph $B_\beta$ such that
\[
\ex(n,B_\beta)=\Theta(n^\beta).
\]    
\end{conjecture}

Although Conjecture~\ref{Conj:Rational-Exponents-Conjecture} has been verified for certain special values of $\beta$ ~\cite{ConlonJanzer2022,CJL2021,JiangMaYepremyan2022,JiangQiu2020,JiangQiu2023,KangKimLiu2021}, it remains widely open in general. A breakthrough was given by Bukh and Conlon~\cite{BukhConlon}, who confirmed the finite-family version of the conjecture: Every rational $\beta \in (1,2)$ is realized as the exponent of $\ex(n,\mathcal{B}_\beta)$ for some finite family $\mathcal{B}_\beta$. 

In this note, we study the generalized rational exponents conjecture. After several decades of sporadic results, the generalized Tur\'{a}n problem was introduced systematically by Alon and Shikhelman~\cite{AlonShikhelman2016}; see the survey of Gerbner and Palmer~\cite{GerbnerPalmer2026}. Recently, Gerbner and Palmer~\cite{GerbnerPalmer2026} posed the following generalized version of  Conjecture~\ref{Conj:Rational-Exponents-Conjecture}. 

\begin{conjecture} [Generalized rational exponents conjecture]\label{Conj:Generalized-rational-exponents-conjecture}
For every rational number $\alpha\ge 1$, there exist graphs $H$ and $F$ such that
\[
\ex(n,H,F)=\Theta(n^\alpha).
\]
\end{conjecture}

Substantial progress toward this conjecture has been made in the setting where a finite forbidden family is allowed. The Bukh--Conlon theorem~\cite{BukhConlon} is precisely the case $H=K_2$. Subsequently, English, Halfpap, and Krueger~\cite{EnglishHalfpapKrueger2025} established the analogous finite-family result when $H$ is a clique, while English and Spiro~\cite{EnglishSpiro2025} obtained general results for arbitrary $H$. In the single-forbidden graph setting, Janzer, Longbrake, and Yepremyan~\cite{JanzerLongbrakeYepremyan2026} proved that every fixed counting graph with at least one edge admits infinitely many realizable exponents.  In this paper, we confirm Conjecture~\ref{Conj:Generalized-rational-exponents-conjecture} in full generality. For a graph $G$, let $\diam(G)$ denote the \emph{diameter} of $G$. 

\begin{theorem}\label{thm:main}
For every rational number $\alpha\ge 1$, there exist fixed graphs $H_\alpha$ with  $\diam(H_\alpha)\le 3$  and $F_\alpha$ such that
\[
\ex(n,H_\alpha,F_\alpha)=\Theta(n^\alpha).
\]
\end{theorem}

The proof begins with the Bukh--Conlon finite family construction~\cite{BukhConlon}. Its first new ingredient is a localization step that turns the edge lower bound into a triangle lower bound in which all triangles pass through one distinguished vertex. The central compression step then encodes an arbitrary finite forbidden family in one forbidden graph. This is not a formal compactness argument: the ordinary compactness conjecture concerns edge counts, whereas generalized counts can change substantially when a finite family is replaced by one graph. Finally, a rooted pendant-leaf construction shifts the exponent by any nonnegative integer. Together, these three transformations yield Theorem~\ref{thm:main}. 

The paper follows this sequence: Section~\ref{sec:BC} records the Bukh--Conlon input in the form needed here; Section~\ref{sec:localization} localizes its edge exponent to triangle counts; Section~\ref{sec:compression} compresses the resulting finite forbidden family into one graph; and Section~\ref{sec:shift} gives the rooted pendant-leaf shift. Section~\ref{sec:proof-main} combines these reductions to prove Theorem~\ref{thm:main}, and Section~7 closes with further remarks and an open problem.

\section{The Bukh--Conlon finite family construction}\label{sec:BC}
In this section, we introduce the finite family construction due to Bukh and Conlon~\cite{BukhConlon}. 
For a graph $G$ and a subset $X\subseteq V(G)$, we denote by $G[X]$ the subgraph of $G$ induced by $X$, and write $G-X:=G[V(G)\setminus X]$. 
For a vertex $v\in V(G)$, let $N_G(v)$ denote the \emph{neighborhood} of $v$, and let $d_G(v):=|N_G(v)|$ be its \emph{degree}. 
As usual, $K_t$ and $P_t$ denote the complete graph and the path on $t$ vertices, respectively. 
For vertex-disjoint graphs $J$ and $J'$, let $J\sqcup J'$ denote their \emph{disjoint union}, and let $J\vee J'$ denote their \emph{join}, obtained by adding all edges between $V(J)$ and $V(J')$. 
Let $\Aut(J)$ be the \emph{automorphism group} of $J$, and let $\emb(J,G)$ denote the number of injective maps $\varphi:V(J)\to V(G)$ such that $uv\in E(J)$ implies $\varphi(u)\varphi(v)\in E(G)$. 
Then
\begin{equation}\label{eq:emb-copies}
\emb(J,G)=|\Aut(J)|\,N(J,G).
\end{equation}
Thus, for a fixed graph $J$, the quantities $\emb(J,G)$ and $N(J,G)$ differ only by a multiplicative constant depending solely on $J$.

We recall the part of the Bukh--Conlon construction needed here. A \emph{rooted tree} $(T, R)$ consists of a tree $T$ together with an independent set $R \subset V(T)$, which we refer to as the \emph{roots}. For a positive integer $p$, the \emph{$p$th power} $\cT^p(T,R)$ of $(T, R)$ is the family of graphs consisting of all possible unions of $p$ distinct labeled copies of $T$, each of which agree on the set of roots $R$. Note that $\cT^p(T,R)$ consists of more than one graph because we allow the unrooted vertices $V(T)\setminus R$ to meet in every possible way. For positive integers $a,b$ with $b\ge a-1$, let $(T_{a,b},R_{a,b})$ be a rooted tree $(T_{a,b},R_{a,b})$ having exactly $a$ unrooted vertices and $b$ edges. The unrooted vertices contain a distinguished path on $a$ vertices. The following is the main result of Bukh and Conlon  (see Lemmas~1.1--1.3 in~\cite{BukhConlon}).

\begin{theorem}\label{thm:BC-lower}
For positive integers $a,b$ with $b\ge a-1$, there is a positive integer $p$ such that
\begin{equation*}
\ex\bigl(n,\cT^p(T_{a,b},R_{a,b})\bigr)=\Theta\bigl(n^{2-a/b}\bigr).
\end{equation*}
\end{theorem}

\begin{prop}\label{prop:BC}
For every rational $\beta\in(1,2)$, there is a finite family $\cB_\beta=\{B_1,\ldots,B_s\}$
such that
\begin{enumerate}[(i)]
\item $\ex(n,\cB_\beta)=\Theta(n^\beta)$;
\item every $B_i$ is connected;
\item every $B_i$ contains $P_4$ as a subgraph.
\end{enumerate}
\end{prop}

\begin{proof}
Write $2-\beta = a_0/b_0$ with positive integers $a_0 < b_0$. Choose $t$ so that $a := t a_0 \ge 4$, and set $b := t b_0$. Let $p$ be an integer provided by Theorem~\ref{thm:BC-lower}, and define
\[
\mathcal{B}_\beta := \mathcal{T}^p(T_{a,b}, R_{a,b}).
\]
Then, by Theorem~\ref{thm:BC-lower},
\[
\ex(n, \mathcal{B}_\beta) = \Theta(n^{2 - a/b}) = \Theta(n^\beta).
\]
Note that each $T_{a,b}$ has $b+1$ vertices in total, of which $a$ are unrooted; hence its root set has size $(b+1)-a = b-a+1 \ge 2$. Moreover, each member $B_i \in \mathcal{B}_\beta$ is obtained as the union of connected copies of $T_{a,b}$ that share this nonempty root set, and is therefore connected. Since the $a$ unrooted vertices of $T_{a,b}$ contain a path on $a \ge 4$ vertices, every constituent copy of $T_{a,b}$ contains a $P_4$, and consequently so does every  $B_i\in \mathcal{B}_\beta$.
\end{proof}

\section{Localizing an edge exponent into triangles}\label{sec:localization}
In this section, we convert an edge exponent into a triangle exponent and arrange that
every triangle in the lower-bound construction contains one distinguished universal vertex.

For a graph $B$,  let $\widehat{B}:=K_1\vee B$  be its \emph{suspension}. Let $\cZ$ be the family of all isomorphism types of graphs that can be written as the union of the edge sets of three triangles $Q_0,Q_1,Q_2$ satisfying
\begin{equation*}
V(Q_0)\cap V(Q_2)=\emptyset,\qquad
V(Q_0)\cap V(Q_1)\ne\emptyset,\qquad
V(Q_1)\cap V(Q_2)\ne\emptyset.
\end{equation*}
The family $\cZ$ is finite, since such a union has at most seven vertices. The following is the main result in this section. 

\begin{theorem}\label{thm:triangle-localization}
Let $1<\beta<2$, and let $\cB=\{B_1,\ldots,B_s\}$ be a finite family of connected graphs such that every $B_i$ contains $P_4$ and
\[
\ex(n,\cB)=\Theta(n^\beta).
\]
Define
\begin{equation*}
\cA:=\cZ\cup\{\widehat{B}_1,\ldots,\widehat{B}_s\}.
\end{equation*}
Then
\[
\ex(n,K_3,\cA)=\Theta(n^\beta).
\]
Moreover, for every sufficiently large $n$, there is an $n$-vertex $\cA$-free graph $X_n$ with a vertex $x_n$ such that
\begin{equation*}
d_{X_n}(x_n)=n-1,\qquad N(K_3,X_n)=\Omega(n^\beta),
\end{equation*}
and every triangle of $X_n$ contains $x_n$. 
\end{theorem}

\begin{proof}
We first prove the upper bound. Let $G$ be an $n$-vertex $\mathcal{A}$-free graph, and let $\Gamma(G)$ denote its triangle-intersection graph: the vertices of $\Gamma(G)$ are the triangles of $G$, with two vertices adjacent if the corresponding triangles share a vertex in $G$. Then bounding the number of triangles in $G$ is equivalent to bounding $|V(\Gamma(G))|$.

Since $G$ is $\cZ$-free, for any 3-vertex path $Q_0Q_1Q_2$ in $\Gamma(G)$, we have $V(Q_0)\cap V(Q_2)\neq \emptyset$. This implies that every component of $\Gamma(G)$ is a clique. Let $\Gamma_1,\ldots,\Gamma_m$ be the components of $\Gamma(G)$, and put
\[
W_j:=\bigcup_{Q\in V(\Gamma_j)}V(Q).
\]
Then  $W_1,\ldots,W_m$ are pairwise disjoint. Consequently,
\begin{equation}\label{eq:Wsum}
\sum_{j=1}^m |W_j|\le n.
\end{equation}

Fix $j\in [m]$ and a triangle $abc\in V(\Gamma_j)$. Since $\Gamma_j$ is a clique, every triangle represented in $\Gamma_j$ meets $\{a,b,c\}$. Observe that the number of triangles containing a fixed vertex $v$ in a graph equals the number of edges in the induced subgraph of the neighborhood of $v$. Thus, 
\begin{equation}\label{eq:Gamma-bound}
|V(\Gamma_j)|\le e\bigl(G[N_G(a)\cap W_j]\bigr)+e\bigl(G[N_G(b)\cap W_j]\bigr)+e\bigl(G[N_G(c)\cap W_j]\bigr).
\end{equation}
For every $v\in V(G)$, the graph $G[N_G(v)]$ is $\cB$-free. By the assumption of the theorem,  there is a constant $C$ such that
\begin{equation}\label{eq:neighborhood-bound}
e(G[Y])\le C|Y|^\beta
\end{equation}
for every $Y\subseteq N_G(v)$. Applying \eqref{eq:neighborhood-bound} to \eqref{eq:Gamma-bound} gives
\[
|V(\Gamma_j)|\le3C|W_j|^\beta.
\]
This together with \eqref{eq:Wsum} and $\beta>1$ yields that 
\[
N(K_3,G)=\sum_{j=1}^m |V(\Gamma_j)|
\le3C\sum_{j=1}^m |W_j|^\beta
\le3C\left(\sum_{j=1}^m |W_j|\right)^\beta
\le3Cn^\beta.
\]
Therefore $\ex(n,K_3,\cA)=O(n^\beta)$.

For the lower bound, let $J_0$ be an $(n-1)$-vertex $\cB$-free graph with
\[
e(J_0)=\Omega(n^\beta).
\]
Take a maximum cut of $J_0$ and delete all edges internal to its two parts. The resulting spanning bipartite graph $J$ is still $\cB$-free,  and satisfies
\[
e(J)\ge\frac12e(J_0)=\Omega(n^\beta).
\]
Add a new vertex $x$ adjacent to every vertex of $J$, and set
\[
X_n:=K_1\vee J.
\]
Then $d_{X_n}(x)=n-1$. Since $J$ is bipartite, it has no triangle. Therefore every triangle of $X_n$ is uniquely of the form $xuv$ with $uv\in E(J)$, and conversely every edge $uv\in E(J)$ yields precisely one such triangle. Thus
\[
N(K_3,X_n)=e(J)=\Omega(n^\beta).
\]

It remains to show that $X_n$ is $\cA$-free. Since every two triangles of $X_n$ meet at $x$,  $X_n$ contains no member of $\cZ$.  Suppose that $X_n$ contains a copy of $\widehat{B}_i$, and let $z$ denote its apex. If $z$ is mapped to $x$, then the remaining vertices form a copy of $B_i$ in $J$, a contradiction. If $z$ is mapped to some $y\in V(J)$, then the image of $B_i$ lies in $X_n[N_{X_n}(y)]$. The set $N_J(y)$ is independent, and $x$ is adjacent to every vertex of $N_J(y)$, so $X_n[N_{X_n}(y)]$ is a star centered at $x$. A star has no $P_4$ as a subgraph, whereas $P_4\subseteq B_i$. This is again a contradiction.
\end{proof}

We conclude this section with the observation that every member of $\mathcal{A}$ is connected, and every edge of any graph in $\mathcal{A}$ lies in a triangle.

\begin{figure}[t]
\centering
\begin{tikzpicture}[scale=0.88, every node/.style={font=\small}]
  \node[draw, ellipse, minimum width=2.25cm, minimum height=1.35cm] (CH) at (0,0) {$C_H\cong K_q$};
  \fill (1.18,0) circle (1.6pt) node[below=2pt] {$c_H$};
  \coordinate (r) at (2.55,0);
  \coordinate (u) at (3.55,0.72);
  \coordinate (v) at (3.55,-0.72);
  \draw (1.18,0)--(r);
  \draw (r)--(u)--(v)--cycle;
  \fill (r) circle (1.6pt) node[below=2pt] {$r$};
  \fill (u) circle (1.6pt) node[above=2pt] {$u$};
  \fill (v) circle (1.6pt) node[below=2pt] {$v$};
  \node at (1.7,-1.25) {(a) the counting graph $H$};

  \node[draw, ellipse, minimum width=2.25cm, minimum height=1.35cm] (C) at (6.0,0) {$C\cong K_q$};
  \fill (7.20,0) circle (1.6pt) node[below=2pt] {$c$};
  \node[draw, ellipse, minimum width=3.00cm, minimum height=1.45cm, align=center] (X) at (10.80,0) {$X_{n-q}$\\[-1pt]{\scriptsize all triangles contain $x$}};
  \draw (7.20,0)--(9.20,0);
  \fill (9.20,0) circle (1.6pt) node[below=2pt] {$x$};
  \node at (8.55,-1.25) {(b) the lower-bound host $G_n$};
\end{tikzpicture}
\caption{The bridge-anchor construction. The unique cross-edge is a bridge and lies in no triangle.}\label{fig:bridge-anchor}
\end{figure}
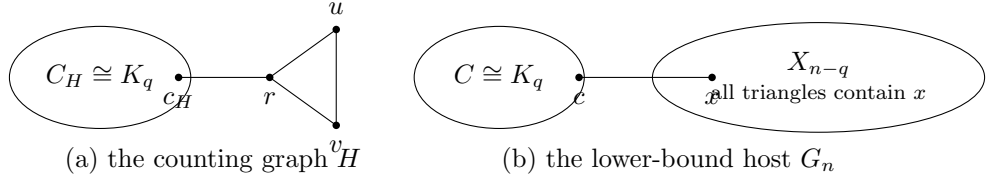

\section{Compressing the forbidden family to one graph}\label{sec:compression}
In this section, we prove the following central reduction theorem.  Its lower-bound construction explains why Theorem~\ref{thm:triangle-localization} was arranged so that every edge of every obstruction lies in a triangle.

\begin{theorem}\label{thm:compression}
Let $\cA=\{A_1,\ldots,A_s\}$ be a finite family of graphs, and let $\beta\in (1,2)$. Assume that
\begin{enumerate}[(i)]
\item every $A_i$ is connected and contains an edge;
\item every edge of every $A_i$ lies in a triangle;
\item $\ex(n,K_3,\cA)=\Theta(n^\beta)$;
\item for every sufficiently large $n$, there is an $n$-vertex $\cA$-free graph $X_n$ with a vertex $x_n$ such that every triangle of $X_n$ contains $x_n$ and
\[
N(K_3,X_n)=\Omega(n^\beta),\qquad d_{X_n}(x_n)=\Omega(n).
\]
\end{enumerate}
Then there exist a connected graph $H$, a single graph $F$, and a distinguished vertex $r\in V(H)$ such that
\[
\ex(n,H,F)=\Theta(n^\beta)\qquad\text{and}\qquad \diam(H)\le3.
\]
Moreover, for every sufficiently large $n$, there is an $n$-vertex $F$-free graph $G_n$ with a vertex $y_n$ satisfying $d_{G_n}(y_n)=\Omega(n)$ and such that $\Omega(n^\beta)$ embeddings of $H$ into $G_n$ send $r$ to $y_n$.
\end{theorem}

\begin{proof}[Proof of Theorem \ref{thm:compression}]
Choose a list $B_1,\ldots,B_M$ of members of $\cA$ in which every $A_i$ occurs at least once and $M\ge2$; repetitions are allowed. Put
\begin{equation*}
L:=\sum_{j=1}^M v(B_j),\qquad q:=L-1,
\end{equation*}
and define the single forbidden graph
\begin{equation*}
F:=B_1\sqcup\cdots\sqcup B_M.
\end{equation*}

Define $H$ as follows. Take a clique $C_H\cong K_q$ and a vertex-disjoint triangle on vertices $r,u,v$. Fix $c_H\in C_H$ and add the single edge $c_Hr$. Thus $c_Hr$ is a bridge. The graph $H$ is connected and has diameter at most $3$. The two graphs used in the lower-bound construction are summarized in Figure~\ref{fig:bridge-anchor}.

\begin{claim}\label{clm:unique-clique}
Every $F$-free graph contains at most one $q$-clique.
\end{claim}
\begin{proof}
Suppose that an $F$-free graph $G$ contains two distinct $q$-cliques $C$ and $D$. Let
\[
\delta:=|C\setminus D|=|D\setminus C|\ge1,\qquad b_1:=v(B_1).
\]
Since $M\ge2$ and every $B_j$ contains an edge, we have $b_1\le L-2<q$. Choose $S\subseteq C$ with $|S|=b_1$ so that $|S\cap D|$ is as small as possible. Since $C\setminus D$ has $\delta$ vertices,
\[
|S\cap D|=\max\{0,b_1-\delta\}.
\]
It follows that
\[
|D\setminus S|=q-\max\{0,b_1-\delta\}
\ge L-b_1=\sum_{j=2}^M v(B_j).
\]
Indeed, if $b_1\le\delta$, then $|D\setminus S|=q=L-1\ge L-b_1$; if $b_1>\delta$, then
\[
|D\setminus S|=L-1-b_1+\delta\ge L-b_1.
\]
Thus, we can find a copy of  $B_1$ in $C[S]$,  vertex-disjoint  copies of $B_2,\ldots,B_M$ in $D\setminus S$. The resulting component-images are mutually vertex-disjoint and form a copy of $F$, a contradiction.
\end{proof}

\begin{claim}\label{clm:remove-clique}
If $C$ is a $q$-clique in an $F$-free graph $G$, then $G-C$ is $\cA$-free.
\end{claim}
\begin{proof}
Suppose that $G-C$ contains a copy of some $A_i$. Choose $j$ with $B_j=A_i$ and use this copy for the $j$th component of $F$. The remaining components require $L-v(A_i)\le L-1=q$ vertices in total, which can be found in $C$, a contradiction.
\end{proof}

We now prove the upper bound. For  sufficiently large $n$, let $G$ be an $n$-vertex $F$-free graph. If $G$ contains no $q$-clique, then $G$ is $H$-free and we are done. Otherwise, by Claim~\ref{clm:unique-clique}, $G$ has a unique $q$-clique $C$. In every embedding of $H$ into $G$, the $q$ vertices of the clique $C_H$ have a $q$-clique as their image, and hence map bijectively onto $C$. The triangle $ruv$, being vertex-disjoint from $C_H$, maps into $G-C$. Hence
\begin{equation}\label{eq:emb-H-G-upperbound}
\emb(H,G)\le q!\,\emb(K_3,G-C).
\end{equation}
By Claim~\ref{clm:remove-clique}, the graph $G-C$ is $\cA$-free. Assumption (iii) gives $N(K_3,G-C)=O(n^\beta)$, and therefore $\emb(K_3,G-C)=6N(K_3,G-C)=O(n^\beta)$. This together with \eqref{eq:emb-copies} and \eqref{eq:emb-H-G-upperbound} gives 
\[
N(H,G)=O(n^\beta).
\]
Thus $\ex(n,H,F)=O(n^\beta)$.

For the lower bound, take the special $\cA$-free graph
\[
X:=X_{n-q}
\]
from assumption (iv), with distinguished vertex $x$. Since $q$ is fixed,
\[
N(K_3,X)=\Omega(n^\beta),\qquad d_X(x)=\Omega(n).
\]
Take a vertex-disjoint clique $C\cong K_q$, choose $c\in V(C)$, and add exactly one edge between $C$ and $X$, namely $cx$. Denote the resulting $n$-vertex graph by $G_n$, and put $y_n:=x$. Note that $cx$ lies in no triangle of $G_n$. 

We claim that $G_n$ is $F$-free. Suppose otherwise, and consider a connected component $B_j$ of a copy of $F$. If its image met both $C$ and $X$, connectivity supplies a path in $B_j$ whose endpoints map to different sides. The first edge of this path whose image changes sides must map to the unique cross-edge $cx$. By assumption (ii), this edge of $B_j$ lies in a triangle; injectivity sends that triangle to a triangle of $G_n$ containing $cx$, a contradiction. Thus the image of each $B_j$ lies entirely in $C$ or entirely in $X$. Furthermore, we know that no $B_j$ can lie in $X$, as $X$ is $\cA$-free and $B_j\in\cA$. Hence every component $B_1,\ldots,B_M$ would have to lie in $C$, which is impossible as $v(F)=L=q+1$ but $v(C)=q$.

It remains to count rooted embeddings. Fix a bijection from $C_H$ to $C$ sending $c_H$ to $c$. Every triangle $\{x,a,b\}$ of $X$ gives two embeddings of $H$ into $G_n$ that send
\[
r\mapsto x,\qquad \{u,v\}\mapsto\{a,b\},
\]
according to the two orders of $a,b$. Thus the number of embeddings sending $r$ to $x$ is at least
\[
2N(K_3,X)=\Omega(n^\beta).
\]
Moreover,
\[
d_{G_n}(x)=d_X(x)+1=\Omega(n).
\]
Since $H$ is fixed, \eqref{eq:emb-copies} converts this rooted embedding estimate into $N(H,G_n)=\Omega(n^\beta)$, completing the proof.
\end{proof}

\section{Shifting the exponent at a rooted vertex}\label{sec:shift}
The following reduction is elementary, but its rooted formulation makes the proof of Theorem~\ref{thm:main} modular.

\begin{lemma}\label{lem:rooted-shift}
Let $H$ and $F$ be fixed graphs, let $r\in V(H)$, and let $\beta\in (1,2)$. Assume that
\[
\ex(n,H,F)=O(n^\beta).
\]
Suppose that, for every sufficiently large $n$, there is an $n$-vertex $F$-free graph $G_n$ with a vertex $x_n$ such that
\begin{enumerate}[(i)]
\item $d_{G_n}(x_n)=\Omega(n)$;
\item $\Omega(n^\beta)$ embeddings of $H$ into $G_n$ send $r$ to $x_n$.
\end{enumerate}
For an integer $k\ge0$, let $H^{(k)}$ be obtained from $H$ by adjoining $k$ new leaves, each adjacent only to $r$. Then
\[
\ex(n,H^{(k)},F)=\Theta(n^{\beta+k}).
\]
\end{lemma}

\begin{proof}
For the upper bound, every embedding of $H^{(k)}$ into an $F$-free graph $G$ restricts to an embedding of $H$. By the hypothesis and \eqref{eq:emb-copies}, $\emb(H,G)=O(n^\beta)$. Once this restriction is fixed, the $k$ labeled leaves have at most $n^k$ possible images. Hence
\[
\emb(H^{(k)},G)\le n^k\emb(H,G)=O(n^{\beta+k}),
\]
and \eqref{eq:emb-copies} gives
\[
\ex(n,H^{(k)},F)=O(n^{\beta+k}).
\]

For the lower bound, use $G_n$ and $x_n$. Let $R_n$ be the number of embeddings $\varphi:H\to G_n$ satisfying $\varphi(r)=x_n$. Then $R_n=\Omega(n^\beta)$. For each such $\varphi$, the new leaves may be mapped injectively to distinct vertices of
\[
N_{G_n}(x_n)\setminus\varphi(V(H)).
\]
Thus the number of extensions of $\varphi$ is at least
\[
\bigl(d_{G_n}(x_n)-v(H)\bigr)_k,
\]
where $(m)_0:=1$ and $(m)_k=m(m-1)\cdots(m-k+1)$ for $k\ge1$. Since $d_{G_n}(x_n)=\Omega(n)$ and $H,k$ are fixed, this quantity is $\Omega(n^k)$ for all sufficiently large $n$. Distinct embeddings of $H$ give distinct extensions after restriction, so
\[
\emb(H^{(k)},G_n)\ge R_n\bigl(d_{G_n}(x_n)-v(H)\bigr)_k=\Omega(n^{\beta+k}).
\]
A final application of \eqref{eq:emb-copies} gives the lower bound.
\end{proof}

\section{Proof of the main theorem}\label{sec:proof-main}
\begin{proof}[Proof of Theorem~\ref{thm:main}]
Fix a rational number $\alpha\ge1$. If $\alpha=1$, take $H_\alpha=K_2$ and $F_\alpha=P_3$, and we are done.  If $\alpha=m\ge2$ is an integer, then take $H_\alpha=K_m$ and $F_\alpha=K_{m+1}$. The trivial upper bound $N(K_m,G)\le\binom{n}{m}$ gives $O(n^m)$, while the balanced complete $m$-partite graph yields $N(K_m,G)\ge \Omega(n^m)$. Thus, 
\[
\ex(n,K_m,K_{m+1})=\Theta(n^m).
\]

It remains to consider nonintegral $\alpha>1$. Write
\[
\alpha=m+\theta,\qquad m=\lfloor\alpha\rfloor\ge1,\qquad 0<\theta<1,
\]
and set
\begin{equation}\label{eq:betak}
\beta:=1+\theta\in(1,2),\qquad k:=m-1\ge0.
\end{equation}
Then $\beta$ is rational and
\begin{equation}\label{eq:beta-k-alpha}
\beta+k=\alpha.
\end{equation}

Apply Proposition~\ref{prop:BC} to obtain a finite family $\cB_\beta$ of connected graphs, each containing $P_4$, such that
\[
\ex(n,\cB_\beta)=\Theta(n^\beta).
\]
By Theorem~\ref{thm:triangle-localization}, there is a finite family $\cA_\beta$ such that
\[
\ex(n,K_3,\cA_\beta)=\Theta(n^\beta),
\]
every member of $\cA_\beta$ is connected, and every edge of every member lies in a triangle. The same theorem supplies, for all sufficiently large $n$, an $\cA_\beta$-free graph $X_n$ with $\Omega(n^\beta)$ triangles, all through a universal vertex $x_n$. 

All hypotheses of Theorem~\ref{thm:compression} are therefore satisfied. We obtain a connected graph $H_\beta$, a single forbidden graph $F_\beta$, and a root $r\in V(H_\beta)$ such that
\[
\ex(n,H_\beta,F_\beta)=\Theta(n^\beta),\qquad \diam(H_\beta)\le3,
\]
and such that the rooted lower-bound hypothesis of Lemma~\ref{lem:rooted-shift} holds.

Apply Lemma~\ref{lem:rooted-shift} with the integer $k$ from \eqref{eq:betak}, and set
\[
H_\alpha:=H_\beta^{(k)},\qquad F_\alpha:=F_\beta.
\]
By \eqref{eq:beta-k-alpha},
\[
\ex(n,H_\alpha,F_\alpha)=\Theta(n^{\beta+k})=\Theta(n^\alpha).
\]

The graph $H_\alpha$ is connected. In the construction of Theorem~\ref{thm:compression}, the root $r$ is the triangle-side endpoint of the bridge joining the anchor clique to the triangle. Every new leaf is adjacent to $r$. A new leaf is at distance at most $3$ from every clique vertex and at distance at most $2$ from every triangle vertex; two new leaves are at distance $2$.   Hence we still have  $\diam(H_\alpha)\le3.$
This completes the proof.
\end{proof}

\section{Concluding remarks}

In this note, we prove the generalized rational exponents conjecture. 
The proof builds on the Bukh--Conlon finite family construction. 
We first localize their edge lower bound to rooted triangle counts, then compress the resulting forbidden family into a single forbidden graph. 
Finally, rooted pendant leaves are used to shift the exponent by an arbitrary nonnegative integer.

The forbidden graph produced by Theorem~\ref{thm:compression} is generally disconnected. Its components encode the members of a finite forbidden family, while the anchor clique has exactly one vertex fewer than the total number of vertices needed to host all components. The bridge lies in no triangle, so the hypothesis that every obstruction edge lies in a triangle prevents a connected component from using vertices on both sides. The argument address us to the  following stronger problem.

\begin{ques}\label{ques:connected-F}
For every rational number $\alpha \ge 1$, we ask whether there exist connected graphs $H_\alpha$ and $F_\alpha$ such that
\[
\operatorname{ex}(n, H_\alpha, F_\alpha) = \Theta(n^\alpha).
\]
\end{ques}

\section*{Acknowledgements}
Jianfeng Hou was supported by the National Key R\&D Program of China (Grant No. 2023YFA1010202) and by the Central Guidance on Local Science and Technology Development Fund of Fujian
 Province (Grant No. 2023L3003). Caihong Yang was supported by the Scientific Research
Funds at China University of Geosciences (Wuhan) (Project No.2026039).

\bibliographystyle{abbrv}
\bibliography{main}
\end{document}